\documentclass[11pt]{amsart}
\usepackage{amssymb,amsthm}
\usepackage{hyperref}
\usepackage{color}

\numberwithin{equation}{section}
\newtheorem{theorem}{Theorem}[section]
\newtheorem{lemma}[theorem]{Lemma}
\newtheorem{corollary}[theorem]{Corollary}
\newtheorem{quest}[theorem]{Question}
\newtheorem{proposition}[theorem]{Proposition}

\theoremstyle{definition}
\newtheorem{definition}[theorem]{Definition}
\newtheorem{problem}[theorem]{Problem}

\renewcommand{\wr}{\mathop{\mathrm{wr}}}
\newcommand{\Sym}{\mathop{\mathrm{Sym}}}

\begin{document}
\title[On the order of Borel subgroups of group amalgams]{On the order of Borel subgroups of group
amalgams and an application to locally-transitive graphs}

\author[L. Morgan, P. Spiga, G. Verret]{Luke Morgan, Pablo Spiga, Gabriel Verret}

\address{Luke Morgan,  School of Mathematics and Statistics, 
\newline\indent University of Western Australia, 35 Stirling Highway, Crawley, WA 6009, Australia.} 
\email{luke.morgan@uwa.edu.au}

\address{Pablo Spiga,  Dipartimento di Matematica Pura e Applicata, 
\newline\indent University of Milano-Bicocca, Milano, 20125 Via Cozzi 55, Italy.} 
\email{pablo.spiga@unimib.it}

\address{Gabriel Verret,  School of Mathematics and Statistics, 
\newline\indent University of Western Australia, 35 Stirling Highway, Crawley, WA 6009, Australia.
\newline\indent Also affiliated with: FAMNIT, University of Primorska, 
\newline\indent Glagolja\v{s}ka 8, SI-6000 Koper, Slovenia.}
\email{gabriel.verret@uwa.edu.au}

\thanks{The research of the first author is supported by the Australian Research Council grant
DP120100446. The last author is supported by UWA as part of the ARC grant DE130101001.}


\begin{abstract}
A permutation group is called \emph{semiprimitive} if each of its normal subgroups is either
transitive or semiregular. Given nontrivial finite transitive permutation groups $L_1$ and $L_2$
with $L_1$ not semiprimitive, we construct an infinite family of rank two amalgams of permutation
type $[L_1,L_2]$ and Borel subgroups of strictly increasing order. As an application, we show that
there is no bound on the order of edge-stabilisers in locally $[L_1,L_2]$ graphs.

We also consider the corresponding question for amalgams of rank $k\geq 3$. We completely resolve
this by showing that the order of the Borel subgroup is bounded by the permutation type
$[L_1,\dots,L_k]$ only in the trivial case where each of $L_1,\dots,L_k$ is regular.
\end{abstract}

\maketitle

\section{Introduction}
All graphs in this paper  are connected, simple and locally finite. Let $\Gamma$ be a graph, let $v$
be a vertex of $\Gamma$ and let $G$ be a group of automorphisms of $\Gamma$. We denote by
$\Gamma(v)$ the neighbourhood of $v$, by $G_v$ the stabiliser of $v$ in $G$, and by
$G_v^{\Gamma(v)}$ the permutation group induced by $G_v$ on $\Gamma(v)$. We say that $\Gamma$ is
$G$-locally-transitive if $G_v^{\Gamma(v)}$ is transitive for every vertex $v$ of $\Gamma$. (This is
easily seen to imply that $G$ is transitive on the edges of $\Gamma$.)

The starting point for our investigations is a classical result of Goldschmidt~\cite{Goldschmidt}, a
consequence of which states that in a finite $G$-locally-transitive graph of valency three, the
edge-stabilisers have order dividing $128$. Inspired by this result, we introduce the following
terminology.

Let $L_1$ and $L_2$ be finite transitive permutation groups, let $[L_1,L_2]$ denote the multiset
containing $L_1$ and $L_2$  and let $\Gamma$ be a  $G$-locally-transitive graph. We say that
$(\Gamma,G)$ is \emph{locally} $[L_1,L_2]$ if, for some edge $\{u,v\}$ of $\Gamma$, we have
permutation isomorphisms $G_u^{\Gamma(u)}\cong L_1$ and $G_v^{\Gamma(v)}\cong L_2$.

\begin{definition}
The multiset $[L_1,L_2]$ is  \emph{locally-restrictive} if there exists a constant $c \in \mathbb N$
such that, if $\Gamma$ is a finite  $G$-locally-transitive graph with $(\Gamma,G)$ locally
$[L_1,L_2]$ and $\{u,v\}$ is an edge of $\Gamma$, then $|G_{uv}|\leq c$.
\end{definition}

With this terminology, Goldschmidt's result implies that, if $L_1$ and $L_2$ are transitive
permutation groups of degree three then $[L_1,L_2]$ is locally-restrictive. A related conjecture of
Goldschmidt-Sims states that  if $L_1$ and $L_2$ are both primitive permutation groups then
$[L_1,L_2]$ is locally-restrictive. Whilst there has been some progress on the Goldschmidt-Sims
Conjecture (see \cite{fan,fan2,glasner,morgan,rowleytrees,stelledgetrans}), it remains open.
Although the focus of the Goldschmidt-Sims Conjecture is on primitive permutation groups, it is
still possible for $[L_1,L_2]$ to be locally-restrictive even when neither $L_1$ nor $L_2$ is
primitive. For example, it is easy to see that if $L_1$ and $L_2$ are both regular permutation
groups then $[L_1,L_2]$ is locally-restrictive. We therefore pose the following problem.

\begin{problem}\label{problem:main}
Determine which pairs of finite transitive permutation groups are locally-restrictive.
\end{problem}

Our main result, Theorem~\ref{theo:graph}, is a significant step towards solving
Problem~\ref{problem:main}.

\begin{theorem}
\label{theo:graph}
Let $L_1$ and $L_2$ be nontrivial finite transitive permutation groups. If one of $L_1$ or $L_2$ is
not semiprimitive then $[L_1,L_2]$ is not locally-restrictive.
\end{theorem}

(A permutation group is called \emph{semiregular} if the identity is the only element of the group
that fixes a point and  \emph{semiprimitive} if each of its normal subgroups is either transitive or
semiregular.) In view of Theorem~\ref{theo:graph}, we are naturally led to pose the following
question, the answer to which we believe to be positive.

\begin{quest}
\label{question}
Does the converse of Theorem~$\ref{theo:graph}$ hold? In other words, if $L_1$ and $L_2$ are finite
transitive semiprimitive permutation groups, is $[L_1,L_2]$ locally-restrictive?
\end{quest}

Our notion of locally-restrictive is to locally-transitive graphs what the notion of
graph-restrictive (see \cite[Definition 2]{PSVRestrictive}) is to arc-transitive graphs. Many of the
concepts and results we have discussed so far have well-known analogues in the arc-transitive case.
For example, Goldschmidt's Theorem can be seen as the locally-transitive version of Tutte's famous
result on arc-transitive graphs of valency three~\cite{Tutte,Tutte2}. Similarly,  the
Goldschmidt-Sims Conjecture corresponds to the long-standing Weiss Conjecture~\cite{Weiss} 
 which asserts that primitive groups are graph-restrictive. The recent Poto\v{c}nik-Spiga-Verret
Conjecture \cite[Conjecture~3]{PSVRestrictive}  asserts that a permutation group is
graph-restrictive if and only if it is semiprimitive. Remarkable evidence towards this conjecture
can be found in~\cite{PabloGab}, where the intransitive case is dealt with. For recent progress on
the transitive case,  see \cite{giudicimorgan,giudicimorgan2,PPSS,PraegerSV,Pablo}.  

We remark that an affirmative answer to Question~\ref{question} would show the validity of both the
Weiss and Poto\v{c}nik-Spiga-Verret  Conjectures. In fact, these conjectures can easily be rephrased
using our terminology. Indeed, if $(\Gamma,G)$ is locally $[L,\mathrm C_2]$ (where $C_2$ denotes the
cyclic group of order $2$) 
then $ \Gamma $ is simply the barycentric subdivision of a graph $\tilde \Gamma$ on which $G$  acts
faithfully and arc-transitively and such that $(\tilde \Gamma,G)$ is locally $[L,L]$. There is an obvious converse to this procedure,  thus $[L,\mathrm C_2]$ is locally-restrictive if and only if $L$ is graph-restrictive.

%
%
The theory of groups acting on trees due to Bass-Serre allows us to interpret
Question~\ref{question}  in terms of  locally-transitive discrete subgroups of the automorphism
group of a bi-regular tree. This equivalence will be proved in Section~\ref{sec:graph}. For now, let
us simply point out that, under this equivalence,  Theorem~\ref{theo:graph} is equivalent to
Theorem~\ref{cor:discsubgroups} below.

\begin{theorem}
\label{cor:discsubgroups}
Let $L_1$ and $L_2$ be nontrivial finite transitive permutation groups and let $\mathfrak{T}$ be the
bi-regular tree with valencies the degrees of $L_1$ and $L_2$. If one of $L_1$ and $L_2$ is not
semiprimitive then, for every integer $c$, there exists a group $G$ of automorphisms of
$\mathfrak{T}$ such that $(\mathfrak{T},G)$ is locally $[L_1,L_2]$  and $c\leq |G_{uv}|<\infty$ for
some edge $\{u,v\}$ of $\mathfrak{T}$.
\end{theorem}

 Theorem ~\ref{cor:discsubgroups} is a significant improvement of  \cite[(7.14)]{basskulkarni} which
shows that the automorphism group of a bi-regular tree with composite valencies contains strictly
ascending chains of locally-transitive discrete subgroups, but with no control over the local
permutation groups.

Question~\ref{question} also has a natural formulation in terms of group amalgams of rank two.
Before presenting our result using this language, we first define amalgams, following~\cite{ivsp}.

\begin{definition}
Let $k\geq 2$. A \emph{rank $k$ amalgam} is a finite set  $\mathcal A$  together with a set of $k$ subsets
$P_1,\ldots,P_k$, where each $P_i$ forms a group, $\bigcup_{i=1}^k P_i=\mathcal A$, $\bigcap_{i=1}^k
P_i\neq\emptyset$ and, for every  $i,j\in\{1,\ldots,k\}$ the group operations defined on $P_i$ and
$P_j$ coincide when restricted to $P_i \cap P_j$. 

The \emph{Borel subgroup}  of $\mathcal A$  is $\bigcap_{i=1}^k P_i $ and is denoted
$\mathcal{B}(\mathcal A)$. If there is no nontrivial subgroup of $\mathcal{B}(\mathcal A)$ that is
normalised by each of $P_1,\ldots,P_k$ then we say that $\mathcal A$ is \emph{faithful}. The
\emph{permutation type} of $\mathcal A$  is the multiset $[L_1,\dots,L_k]$ where $L_i$ is the
permutation group induced by $P_i$ in its action on the right cosets of $\mathcal{B}(\mathcal A)$ in
$P_i$.
\end{definition}

In Section~\ref{sec:graph} we will show the equivalence between Theorem~\ref{theo:graph} and
Theorem~\ref{theo:main} below.

\begin{theorem}\label{theo:main}
Let $L_1$ and $L_2$ be nontrivial finite transitive permutation groups. If one of $L_1$ and $L_2$ is
not semiprimitive then, for every integer $c$, there exists a rank two faithful amalgam of
permutation type $[L_1,L_2]$ with Borel subgroup of order at least $c$.
\end{theorem}

Our proof of Theorem~\ref{theo:main} is constructive and can be found in Section~\ref{Sec:ResAreSP}.
The construction used in the proof is a generalisation of the construction that appeared in
\cite[Section~4]{PSVRestrictive} (in fact, Theorem~\ref{theo:graph} generalises \cite[Theorem
4]{PSVRestrictive}) which in turn was inspired by the so-called wreath extension construction
\cite[Chapter IV, 8.1]{malle}. A precursory idea to this construction can also be traced
to~\cite[Section~$4$]{Pablo1}.   

Theorem~\ref{theo:main} naturally leads one to wonder about the corresponding statement for amalgams
of rank greater than two. Here is the complete answer in this case.

\begin{theorem}\label{theo:mainmain}
Let $k\geq 3$ and let $L_1,\ldots,L_k$ be nontrivial finite transitive permutation groups. The
following are equivalent:
\begin{enumerate}
\item One of $L_1,\ldots,L_k$ is not regular.
\item For every integer $c$, there exists a rank $k$ faithful amalgam of permutation type
$[L_1,\ldots,L_k]$  with Borel subgroup of order at least $c$.
\end{enumerate}
\end{theorem}

%
In fact, it is easy to see that if $L_1,\ldots,L_k$ are regular then a faithful amalgam $\mathcal A$
of permutation type $[L_1,\ldots,L_k]$ must have $\mathcal{B}(\mathcal A)=1$ (see
Section~\ref{sec4}). The real meat of Theorem~\ref{theo:mainmain} is therefore the statement that
for rank at least three, these trivial examples are the only ones which admit upper bounds on
$|\mathcal{B}(\mathcal A)|$ depending  upon the permutation type alone. This is in sharp contrast
with the situation in the rank two case; Goldschmidt's result has the highly  nontrivial consequence
that a rank two  faithful amalgam of permutation type $[\Sym(3),\Sym(3)]$ has Borel subgroup of
order at most $128$. In particular, the na\"{i}ve ``$k=2$'' version of Theorem~\ref{theo:mainmain}
is false. We find the relative simplicity of the higher rank case  rather surprising.  Note that
under additional assumptions, the order of the Borel subgroup of a rank three amalgam can be
bounded, see for example~\cite{stelltime}.

\section{Equivalence of Theorems~\ref{theo:graph},~\ref{cor:discsubgroups} and~\ref{theo:main}}
\label{sec:graph}

\begin{lemma}\label{newLemma}
Let $L_1$ and $L_2$ be finite transitive permutation groups and let $B$ be a finite group. The
following are equivalent:
\begin{enumerate}
\item There exists a rank two faithful amalgam of permutation type $[L_1,L_2]$ with Borel subgroup
$B$.
\item There exists a locally $[L_1,L_2]$ pair $(\mathfrak{T},G)$ such that $\mathfrak{T}$ is an
infinite tree and $G_{uv}=B$ for some edge $\{u,v\}$ of $\mathfrak{T}$.
\item There exists a locally $[L_1,L_2]$ pair $(\Gamma,H)$ such that $\Gamma$ is finite and
$H_{uv}=B$ for some edge $\{u,v\}$ of $\Gamma$.
\end{enumerate}
\end{lemma}
\begin{proof}
$(1)\Longrightarrow  (2):$ Let $\mathcal A$ be a rank two faithful amalgam of permutation type
$[L_1,L_2]$ with Borel subgroup $B$. Let $P_1$ and $P_2$ be the two groups involved in $\mathcal A$
and let $G = P_1 *_B P_2$ (that is, $G$ is the free product of $P_1$ and $P_2$ amalgamated over
$B$). By \cite[I.4.1, Theorem~7]{Serre}, there exists an infinite tree $\mathfrak T$ on which $G$
acts faithfully, edge- but not vertex-transitively, and an edge $\{u,v\}$ of $\mathfrak T$ such that
$G_{uv}=B$, $G_u=P_1$ and $G_v=P_2$. As $\mathfrak T$ is $G$-edge- but not $G$-vertex-transitive, it
must be $G$-locally-transitive. Since $\mathcal A$ is of permutation type $[L_1,L_2]$, it follows
that $(\mathfrak T, G)$ is locally $[L_1,L_2]$.

$(2)\Longrightarrow  (3):$ Let $G'$ be the largest subgroup of $G$ that preserves the bipartition of
$\mathfrak{T}$. Note that $(\mathfrak{T},G')$ is locally $[L_1,L_2]$ and $G'_{uv}=B$. By replacing
$G$ with $G'$, we may thus assume that $G$ is not vertex-transitive. By~\cite[I.4.1,
Theorem~6]{Serre}, it follows that $G$ is  isomorphic to $G_u *_{B} G_v$.

By \cite[I.4.1, Proof of Theorem~7]{Serre}, we may assume that the vertex set of $\mathfrak T$ is
the disjoint union of the right coset spaces $G/G_u$ and  $G/G_v$, with two vertices being adjacent
if they have nonempty intersection, and that the action of $G$ on $\mathfrak T$ is given by right
multiplication. In particular, $G_u$ and $G_v$ are adjacent when viewed as vertices of $\mathfrak
T$. Since $\mathfrak T$ is $G$-locally-transitive it follows that the neighbourhood of $G_u$ is
$\{G_v g\mid g\in G_u\}$.

Let $X=G_uG_vG_u\cup G_vG_uG_v$. Note that, as $L_1$, $L_2$ and $B$ are finite, so are $G_u$, $G_v$
and $X$. By \cite[Theorem~2]{Baumslag}, $G$ is residually finite and hence there exists a  normal
subgroup $R$ of finite index in $G$ with $R \cap X = \{1\}$. Let $H = G / R$ and let $\Gamma$  be
the normal quotient graph $\mathfrak{T}/ R$. (The vertices of $\Gamma$ are the $R$-orbits on the
vertex set of $\mathfrak{T}$, with two such $R$-orbits adjacent in $\Gamma$ if there is an edge
between them in $\mathfrak{T}$.) Note that $\Gamma$ is $H$-locally-transitive and finite.

Since $\mathfrak T$ is locally-transitive, it is bi-regular. We now show that $\Gamma$ is bi-regular
with the same valencies as $\mathfrak{T}$. We argue by contradiction and suppose, without loss of
generality, that the $R$-orbit of $G_u$, viewed as a vertex of $\Gamma$, has valency strictly less
than $G_u$, viewed as a vertex of $\mathfrak{T}$. It follows from the definition of $\Gamma$ that
the vertex $G_u$ of $\mathfrak{T}$ must have two distinct neighbours in the same $R$-orbit. Recall
that  the neighbourhood of $G_u$ in $\mathfrak{T}$ is $\{G_v g\mid g\in G_u\}$. In particular, there
exist $g,h\in G_u$ and $r\in R$ such that $G_vg\neq G_vh$ and $G_vgr=G_vh$. This implies that $r\in
G_uG_vG_u \subseteq X$ and hence $r \in R \cap X = \{1\}$, which is a contradiction.

Let $K/R$ be the kernel of the action of $H=G/R$ on $\Gamma$. By the previous paragraph, $\Gamma$ is
bi-regular with the same valencies as $\mathfrak{T}$ and a standard argument yields that $K$ is
semiregular on $\mathfrak{T}$. In particular, $K=R$ (and $\mathfrak{T}$ is a regular cover of
$\Gamma$) and $H$ acts faithfully on $\Gamma$. It follows that the stabiliser in $H$ of the vertex
$G_uR$ of $\Gamma$ is $G_uR/R\cong G_u$, the stabiliser of the vertex $G_vR$ is  $G_vR/R\cong G_v$
and the stabiliser of the edge $\{G_uR,G_vR\}$ is $BR/R\cong B$. Since $(\mathfrak{T},G)$ is locally
$[L_1,L_2]$, this implies that $(\Gamma,H)$ is locally $[L_1,L_2]$.

$(3)\Longrightarrow  (1):$ Let $\mathcal A$ be the rank two  amalgam of the groups $H_u$ and $H_v$
with Borel subgroup $H_{uv}$. Since $\Gamma$ is $H$-locally-transitive, the group generated by $H_u$
and $H_v$ is transitive on edges of $\Gamma$. In particular, any subgroup of $H_{uv}$ that is
normalised by both $H_u$ and $H_v$ must be trivial. This shows that $\mathcal A$ is faithful.
Clearly, $\mathcal A$ has permutation type $[L_1,L_2]$.
\end{proof}

The following is an immediate corollary to Lemma~\ref{newLemma}.

\begin{corollary}\label{newCor}
Let $L_1$ and $L_2$ be finite transitive permutation groups. The following are equivalent:
\begin{enumerate}
\item For every integer $c$, there exists a rank two faithful amalgam of permutation type
$[L_1,L_2]$ with Borel subgroup of order at least $c$.
\item For every integer $c$, there exists a locally $[L_1,L_2]$ pair $(\mathfrak{T},G)$ such that
$\mathfrak{T}$ is an infinite tree and $c\leq |G_{uv}|<\infty$ for some edge $\{u,v\}$ of
$\mathfrak{T}$.
\item $[L_1,L_2]$ is not locally-restrictive.
\end{enumerate}
\end{corollary}

The equivalence of Theorems~\ref{theo:graph},~\ref{cor:discsubgroups} and~\ref{theo:main} follows
immediately from Corollary~\ref{newCor}.

\section{Proof of Theorem~\ref{theo:main}}\label{Sec:ResAreSP}
All groups mentioned in the next two sections are finite. We adopt the notation and hypothesis of
Theorem~\ref{theo:main} and, without loss of generality, we assume that $L_1$ is not semiprimitive.
To simplify notation, we write $L=L_1$ and $R=L_2$. Let $m_2$ be the degree of $R$, let $\ell$ be a
positive integer, let $m=\ell m_2$ and let $\Omega=\{(y,z)\mid 1\leq y\leq m_2,1\leq z\leq \ell \}$.
Observe that $|\Omega|=\ell m_2=m$ and that the action of $R$ on $\{1,\ldots,m_2\}$ induces an
action of $R$ on $\Omega$: for $r\in R$ and $(y,z)\in \Omega$, we set $$(y,z)^{r}=(y^r,z).$$ 

We endow the set $\Omega$ with its natural lexicographic order, that is $(y,z)<(y',z')$ if either
$y<y'$, or $y=y'$ and $z<z'$. This total ordering allows us to identify $\Omega$ with
$\{0,\ldots,m-1\}$ in a natural way : $(1,1)$ is identified with $0$, $(m_2,\ell)$ with $m-1$, etc.
We extend the action of $R$ on $\Omega=\{0,\ldots,m-1\}$  to an action of $R$ on $\{0,\ldots,m\}$ by
letting the point $m$ be fixed by every element of $R$.

Since $L$ is not semiprimitive, there exists  a normal subgroup $K$ of $L$ that is neither
transitive nor semiregular. Denote by $\Delta$ the set of orbits of $K$ and let $K'$ be the kernel
of the action of $L$ on $\Delta$. Note that $K'$ is a normal subgroup of $L$ having the same orbits
as $K$ that is neither transitive nor semiregular. We may thus assume that $K=K'$ without loss of
generality. Let $S$ denote the permutation group induced by the action of $L$ on $\Delta$ and let
$\pi: L \rightarrow S$ be the canonical projection with kernel $K$. Fix $\delta\in \Delta$ and
${\lambda}\in \delta$. Since $K$ is transitive on $\delta$, we have $L_{\delta}=KL_{\lambda}$ and
$S_{\delta}\cong L_{\delta} /K=KL_{\lambda} /K\cong L_{\lambda}/(K\cap
L_{\lambda})=L_{\lambda}/K_{\lambda}$.  We sometimes denote by $\pi$ the restriction 
$\pi|_{L_{\lambda}}:L_{\lambda}\to S_{\delta}$, slightly abusing notation.

Fix $\mathcal{T}$ a transversal for the set of right cosets of $S_{\delta}$ in $S$ with $1\in
\mathcal{T}$. For every $s\in S$, there exists a unique element of $\mathcal{T}$, which we denote by
 $s^\tau$, such that $S_{\delta} s=S_{\delta} s^\tau$. The correspondence $s\mapsto s^\tau$ defines
a map $\tau:S\to \mathcal{T}$ with $1^\tau=1$. 

\begin{lemma}\label{silly1}
If $x,s\in S$, then $(xs^{-1})^\tau s (x^\tau)^{-1}\in S_{\delta}$.
\end{lemma}
\begin{proof}
We have $S_{\delta} xs^{-1}=S_{\delta}(xs^{-1})^\tau$ and hence $S_{\delta} x=S_{\delta}
(xs^{-1})^\tau s$. Furthermore, as $S_{\delta} x=S_{\delta} x^{\tau}$, we obtain $S_{\delta}=
S_{\delta} (xs^{-1})^\tau s(x^{\tau})^{-1}$.
\end{proof}

Let $V$ be the set of all functions from $\Delta$ to $L_{\lambda}$.  Under point-wise
multiplication, $V$ is a group isomorphic to $L_{\lambda}^{\Delta}$. Given $f\in V$ and $g\in L$,
let $f^g$ be the element of $V$ defined by 
\begin{equation}\label{eq:def}
f^g(\sigma)=f\left(\sigma^{(g^\pi)^{-1}}\right),\ \sigma\in\Delta.
\end{equation}
This defines a group action of $L$ on $V$ and the semidirect product $V\rtimes L$ is isomorphic to
the standard wreath product $L_{\lambda}\wr_\Delta L$. Moreover, by extending this action of $L$ on
$V$ to the component-wise action of $L$ on $V^m$, we obtain a semidirect product $L\ltimes V^m$
where the multiplication is given by 
\begin{eqnarray}\label{eq:binary}
(g,f_1,\ldots,f_m)(g',f_1',\ldots,f_m')&=&(gg',f_1^{g'}f_1',\ldots, f_m^{g'}f_m').
\end{eqnarray}

We now isolate some subgroups of $L\ltimes V^m$ that  provide the backbone for our construction.

\begin{definition}\label{A}
We define the following subsets of $L\ltimes V^m$:
\begin{align*}
&A= && \left \{ (g,f_1,\ldots,f_m)\in L\ltimes V^m\begin{array}{ c | c }   & 
(f_i(\delta^x))^\pi=(x(g^\pi)^{-1})^\tau g^\pi(x^\tau)^{-1} \\ 
& \textrm{ for every } i\in \{1,\ldots,m\} \textrm{ and for every }x\in S    \end{array} \right
\},\\\nonumber
& C  =& & \{(g,f_1,\ldots,f_m)\in A\mid g\in L_{\lambda}\}, \\\nonumber
&M = &&\{(1,f_1,\ldots,f_m)\in L\ltimes V^m    \mid f_i(\Delta) \subseteq K_{\lambda} \textrm{ for
every }i\in\{1,\ldots,m\}\}.
\end{align*}
Let $\varphi:A\to L$ be the map defined by $\varphi : (g,f_1,\ldots,f_m) \mapsto g$.
\end{definition}

Note that, by Lemma~\ref{silly1}, the element $(x(g^\pi)^{-1})^\tau g^\pi(x^\tau)^{-1}$ in the
definition of $A$ lies in $S_{\delta}$. 

\begin{lemma}\label{silly2}
The set $A$ is a subgroup of $L\ltimes V^m$. 
\end{lemma}
\begin{proof}
Let $(g,f_1,\ldots,f_m),(g',f_1',\ldots,f_m')\in A$ and let $x\in S$. For every $i\in
\{1,\ldots,m\}$, we have 
\begin{eqnarray*}
((f_i^{g'}f_i')(\delta^x))^\pi&=&(f_i^{g'}(\delta^x))^\pi\cdot(f_i'(\delta^x))^\pi\\
&\overset{(\ref{eq:def})}{=}&(f_i({\delta^x}^{(g'^\pi)^{-1}}))^\pi\cdot(f_i'(\delta^x))^\pi\\
&\overset{\rm{Def.}~\ref{A}}{=}&(((x(g'^\pi)^{-1})(g^\pi)^{-1})^\tau g^\pi
((x(g'^\pi)^{-1})^\tau)^{-1})\cdot((x(g'^\pi)^{-1})^\tau g'^\pi(x^\tau)^{-1})\\
&=&(x((gg')^\pi)^{-1})^\tau (gg')^\pi(x^{\tau})^{-1}.
\end{eqnarray*}
Using~$(\ref{eq:binary})$ and Definition~\ref{A}, this shows that
$(g,f_1,\ldots,f_m)(g',f_1',\ldots,f_m')\in A$. Clearly, the identity of $L\ltimes V^m$ is contained
in $A$. Since $L\ltimes V^m$ is a finite group, this concludes the proof.
\end{proof}

\begin{lemma}\label{silly3}
The map $\varphi$ is a surjective homomorphism.
\end{lemma}
\begin{proof}
By~$(\ref{eq:binary})$, $\varphi$ is a homomorphism.   For each $s\in S_{\delta}$, choose an element
$s^\varepsilon$ of $L_{\lambda}$ with $(s^\varepsilon)^\pi=s$. Since $\pi:L_{\lambda}\to S_{\delta}$
is surjective, $\varepsilon:S_{\delta}\to L_{\lambda}$ is well-defined. Let $g\in L$ and define
$f_{g}:\Delta\to L_{\lambda}$ with 
$$f_g(\delta^x)=((x(g^\pi)^{-1})^{\tau}g^\pi(x^{\tau})^{-1})^\varepsilon,$$ 
for $x\in S$. First, note that, by Lemma~\ref{silly1},
$(x(g^\pi)^{-1})^{\tau}g^\pi(x^{\tau})^{-1}\in S_{\delta}$ and thus $f_g(\delta^x)\in L_{\lambda}$.
To see that $f_g$ is well-defined, note that for every $y\in S_{\delta}$, we have
$(yx)^\tau=x^{\tau}$ and $(yx(g^\pi)^{-1})^\tau=(x(g^\pi)^{-1})^\tau$, and hence 
$f_g(\delta^x)=f_g(\delta^{yx})$. Now
$$(f_g(\delta^x))^\pi=(((x(g^\pi)^{-1})^{\tau}g^\pi(x^{\tau})^{-1})^\varepsilon)^\pi=(x(g^\pi)^{-1})
^{\tau}g^\pi(x^{\tau})^{-1},$$ 
and hence $(g,f_g,\ldots,f_g)\in A$ and $(g,f_g,\ldots,f_g)^\varphi=g$, which concludes the proof. 
\end{proof}

\begin{lemma}\label{lemma:yes}
The kernel of $\varphi$ is $M$ and $M\cong K_{\lambda}^{|\Delta|m}$.
\end{lemma}
\begin{proof}
It is clear that $M\cong K_{\lambda}^{|\Delta|m}$. Suppose first that $(g,f_1,\ldots,f_m)$ is in the
kernel of $\varphi$ then $g=(g,f_1,\ldots,f_m)^\varphi=1$. For every $i\in \{1,\ldots,m\}$ and every
$x\in S$, it follows by Definition~\ref{A} that
$(f_i(\delta^x))^\pi=(x(g^\pi)^{-1})^{\tau}g^\pi(x^\tau)^{-1}=1$ and thus $f_i(\delta^x)\in
K_{\lambda}$. It follows that $(g,f_1,\ldots,f_m)\in M$.

Conversely, if $(g,f_1,\ldots,f_m)\in M$ then $g=1$ and $f_i(\delta^x)\in K_{\lambda}$ for every
$i\in \{1,\ldots,m\}$ and every $x\in S$ and thus 
$(f_i(\delta^x))^\pi=1=(x(g^\pi)^{-1})^{\tau}g^\pi(x^\tau)^{-1}$. In particular,
$(g,f_1,\ldots,f_m)\in A$ and hence $(g,f_1,\ldots,f_m)$ is in the kernel of $\varphi$.
\end{proof}

\begin{lemma}\label{silly4}
The set $C$ is a subgroup of $A$, $M$ is the core of $C$ in $A$ and the permutation group induced by
the action of $A$ on the right cosets of $C$ is permutation isomorphic to $L$.
\end{lemma}
\begin{proof}
By Lemmas~\ref{silly3} and~\ref{lemma:yes}, $\varphi$ is a surjective homomorphism with kernel $M$.
In particular, $M\unlhd A$ and $A/M\cong L$. Note that $C$ is the pre-image of $L_{\lambda}$ under
$\varphi$ and thus $M\leq C\leq A$. As $C^{\varphi}=L_{\lambda}$ and $L_{\lambda}$ is core-free in
$L$, it follows that $M$ is the core of $C$ in $A$. Finally, the action of $A$ on the right cosets
of $C$ is permutation isomorphic to the action of $A^{\varphi}=L$ on the right cosets of
$C^{\varphi}=L_\lambda$, that is, to $L$.
\end{proof} 

We now introduce an alternative notation for the elements of $A$ that will simplify some later
computations. Let $a=(g,f_1,\ldots,f_m)\in A$. For $i\in\{1,\ldots,m\}$, we write $g_i=f_i(\delta)$
and $h_{i-1}={f_i}|_{ \Delta\setminus\{\delta\}}$. (Note that $f_i$ is completely determined by
$(g_i,h_{i-1})$.) We also write $g_0=g$ and then denote $a$ by 
\begin{equation*}
((g_0,\ldots,g_m),(h_0,\ldots,h_{m-1})).
\end{equation*}
Note that, with this notation, the multiplication is not component-wise (in contrast to
$(\ref{eq:binary})$): indeed, if $a'=((g_0',\ldots,g_m'),(h_0',\ldots,h_{m-1}'))$ is another element
of $A$ then
\begin{equation}\label{eq:gogo3}
aa'=((g_0g_0',g_1^{g_0'}g_1'\ldots,g_m^{g_0'}g_m'),(h_0^{g_0'}h_0',\ldots,h_{m-1}^{g_0'}h_{m-1}')).
\end{equation}
 Using the above notation, for each $r\in R$ and $c=((g_0,\ldots,g_m),(h_0,\ldots,h_{m-1}))\in C$, 
let
\begin{equation} c^r = ((g_{0r^{-1}},g_{1 r^{-1}}, \dots, g_{m r^{-1}}),
(h_{0r^{-1}},\dots,h_{(m-1)r^{-1}})),  \label{LA}\end{equation}
where, for $i\in \{0,\dots,m\}$, we denote the image of $i$ under $r$ by $ir$.

\begin{lemma}\label{lemma:action}
Equation~\eqref{LA} defines a group action of $R$ on the group $C$.
\end{lemma}
\begin{proof}
In this proof, it is convenient to use both notations for elements of $C$.  Let $c =
(g,f_1,\dots,f_m)=((g_0,\ldots,g_m),(h_0,\ldots,h_{m-1})) \in C$. Since $c \in C$, we have $g_0=g
\in L_\lambda$ and hence $g^\pi \in S_\delta$. Since $c \in A$, for every $x\in S_\delta$ and $i\in
\{1,\dots,m\}$, we have $x^\tau=1=(x (g^\pi)^{-1})^\tau$ and thus
\begin{equation}  g_i^\pi =(f_i(\delta))^\pi = (f_i(\delta^x))^\pi \overset{\rm{Def.}~\ref{A}}{=} (x
(g^\pi)^{-1})^\tau g^\pi (x^\tau)^{-1} =  g^\pi.\label{STAR} \end{equation}

Let $r \in R$ and write $v=r^{-1}$ and $c^r = (g_{0v},f_1',\dots,f_m')$. We first show that $c^r\in
C$. For every $x\in S_\delta$ and $i\in \{1,\dots,m\}$, we have 
$$(f_i'(\delta^x))^\pi = (f_i'(\delta))^\pi \overset{(\ref{LA})}{=} (g_{iv})^\pi 
\overset{(\ref{STAR})}{=} g_{0v}^\pi = (x (g_{0v}^\pi)^{-1})^\tau g_{0v}^\pi (x^\tau)^{-1} $$
where in the last equality  $x,g_{0v}^\pi\in S_\delta$ is used. Similarly, for every $x\in S
\setminus S_\delta$ and $i\in \{1,\dots,m\}$, we have
$$(f_i'(\delta^x))^\pi \overset{(\ref{LA})}{=} (h_{(i-1)v}(\delta^x))^\pi
\overset{\rm{Def.}~\ref{A}}{=}  (x (g^\pi)^{-1})^\tau g^\pi (x^\tau)^{-1} \overset{(\ref{STAR})}{=}
(x (g_{0v}^\pi)^{-1})^\tau g_{0v}^\pi (x^\tau)^{-1}.$$
This shows that $c^r \in A$. Since $g_i \in L_\lambda$ for all $i \in \{0,\dots,m\}$, we have that
$g_{0v}\in L_\lambda$ and thus $c^r \in C$. 
%
%
%
%
Let $d=((y_0,y_1,\dots,y_m),(z_0,\dots,z_{m-1})) \in C$. Recall that $y_0\in L_\lambda$. Hence, for
$j \in \{1,\dots,m\}$, we have
\begin{equation}  
g_j^{y_0} =f_j^{y_0}(\delta)\overset{(\ref{eq:def})}{=} f_j(\delta^{(y_0^\pi)^{-1}}) = f_j(\delta)=
g_j. \label{STARSTAR} 
\end{equation}
Now, 
\begin{eqnarray*}
cd &\overset{(\ref{eq:gogo3})}{=} &((g_0y_0, g_1^{y_0}y_1,\ldots,g_m^{y_0}y_m),(h_0^{y_0}
z_0,\ldots,h_{m-1}^{y_0} z_{m-1}))\\
&\overset{(\ref{STARSTAR})}{=}& ((g_0y_0,g_1y_1,\ldots,g_m y_m),(h_0^{y_0} z_0,\ldots,h_{m-1}^{y_0}
z_{m-1}))
\end{eqnarray*}
and thus
$$(cd)^r = ((g_{0v} y_{0v},g_{1v}y_{1v},\ldots,g_{mv} y_{mv}),(h_{0 v}^{y_0}
z_{0v},\ldots,h_{(m-1)v}^{y_0} z_{(m-1)v})).$$

Recall that $c^r= ((g_{0v},\ldots, g_{m v}), (h_{0v},\ldots, h_{(m-1)v}))= (g_{0v},f_1',\dots,f_m')$
and thus $f_i'(\delta) = g_{iv}$ and $f_i'(\sigma) = h_{(i-1)v}(\sigma)$ for $i \in \{1,\dots,m\}$
and $\sigma \in \Delta \setminus \{\delta\}$. Similarly, recall that $d^r = ((y_{0v},\ldots, y_{m
v}), (z_{0v},\ldots, z_{(m-1)v}))$ and write $d^r=(y_{0v},e_1',\dots,e_m')$. 
By~$(\ref{eq:binary})$, we have $c^rd^r = (g_{0v}y_{0v},f_1'^{y_{0v}} e_1',\dots,f_m'^{y_{0v}}
e_m')$. Since $y_{0v} \in L_\lambda$, we have $y_{0v}^\pi \in S_\delta$ and hence, for $i\in
\{1,\dots,m\}$, we have
$$ (f_i'^{y_{0v}} e_i') (\delta) = f_i'^{y_{0v}}(\delta)e_i'(\delta) \overset{(\ref{eq:def})}{=}
f_i'(\delta^{(y_{0v}^\pi)^{-1}}) y_{iv} = f_i'(\delta) y_{iv} = g_{iv} y_{iv}.$$
Similarly, for $\sigma \in \Delta \setminus \{ \delta \}$ and $i\in \{1,\dots,m\}$, we have
$$f_{i}'^{y_{0v}}(\sigma) \overset{(\ref{eq:def})}{=} f_{i}'(\sigma^{ (y_{0v}^\pi)^{-1}})= h_{(i-1)
v}(\sigma^{ (y_{0v}^\pi)^{-1}}) \overset{(\ref{STAR})}{=} h_{(i-1) v}(\sigma^{ (y_0^\pi)
^{-1}})\overset{(\ref{eq:def})}{=} h_{(i-1) v}^{y_{0}}(\sigma)$$
and thus $(f_{i}'^{y_{0v}}e_i')(\sigma)=(h_{(i-1) v}^{y_{0}}z_{(i-1)v})(\sigma)$. This shows that
$(cd)^r = c^r d^r$. 

It is clear from (\ref{LA}) that $(c^r)^{r'}=c^{rr'}$ for every $r'\in R$ and $c^r=1$ if and only if
$c=1$, which concludes the proof. 
\end{proof}

By Lemma~\ref{lemma:action}, we can define the semidirect product $C\rtimes R$. Let 
\begin{align*}
&P_2=C\rtimes R,\\
&B=C\rtimes R_1,
\end{align*}
with $B$  viewed as a subgroup of $P_2$. From our definitions we have:

\begin{lemma}\label{retard}
The core of $B$ in $P_2$ is $C$. Moreover, the permutation group induced by the action of $P_2$ on
the right cosets of $B$ is permutation isomorphic to $R$. 
\end{lemma}

From $(\ref{LA})$, $R_1$ inherits an action on $C$ from $R$. We extend this to an action of $R_1$ on
$A$ in the following way:
given $a=((g_0,\ldots,g_m),(h_0,\ldots,h_{m-1}))\in A$ and $r \in R_1$, let
$$a^r =((g_{0r^{-1}},g_{1 r^{-1}}, \dots, g_{m r^{-1}}), (h_{0r^{-1}},\dots,h_{(m-1)r^{-1}})).$$
With minor changes, the proof of Lemma~\ref{lemma:action} can be adapted to show that this induces a
group action of $R_1$ on $A$. (It is helpful to notice that for all $r \in R_1$ we have $0r=0$.) Let
$$P_1=A\rtimes R_1.$$
We view $B$ as a subgroup of $P_1$ in the obvious way. (Note that the action of $R$ on $C$ cannot be
extended to an action of $R$ on $A$ in any meaningful way.)

\begin{lemma}\label{moreretarded}
The core of $B$ in $P_1$ is $M\rtimes R_1$ and the action of $P_1$ on the right cosets of $B$ is
permutation isomorphic to $L$.
\end{lemma}   
\begin{proof}
The proof follows with a computation and from Lemma~\ref{silly4}.
\end{proof}

Let $\mathcal{A}$ be the rank two amalgam of the groups $P_1$ and $P_2$ with $\mathcal{B}(\mathcal
A)=P_1\cap P_2=B$. Lemmas~\ref{retard} and~\ref{moreretarded} show that the permutation type of
$\mathcal{A}$ is  $[L_1,L_2]$.
%
%
\begin{proposition}\label{neverends}
The amalgam $\mathcal{A}$ is faithful.
\end{proposition}
\begin{proof}
Let $N$ be a subgroup of $B$ normal in $P_1$ and in $P_2$. We show that $N=1$. By
Lemma~\ref{retard}, the core of $B$ in $P_2$ is $C$ and hence $N\leq C$. By
Lemma~\ref{moreretarded}, the core of $B$ in $P_1$ is $M\rtimes R_1$ and thus $N\leq C\cap (M\rtimes
R_1)=M$.

For $i\in\{0,\ldots,m\}$, let $G(i)$ be the proposition: for every 
$((g_0,\ldots,g_m),(h_0,\ldots,h_{m-1}))\in N$, we have $g_i=1$ . Similarly, for $i
\in\{0,\ldots,m-1\}$, let $H(i)$ be the proposition: for every 
$((g_0,\ldots,g_m),(h_0,\ldots,h_{m-1}))\in N$, we have $h_i=1$ . We prove the following preliminary
claims.

\smallskip
\noindent\textsc{Claim 1. }Let $i\in \{1,\ldots,m\}$ and let $\sigma\in \Delta$. Suppose that, for
every $(1,f_1,\ldots,f_m)\in N$, we have $f_i(\sigma)=1$. Then $G(i)$ and $H(i-1)$ hold.
\smallskip

\noindent 
Let $(1,f_1,\ldots,f_m)\in N$ and let $\mu \in \Delta$. Since $S$ is transitive on $\Delta$ and
$\pi$ is surjective, there exists $g\in L$ such that $\sigma^{ (g^{\pi})^{-1}} = \mu$. By
Lemma~\ref{silly3}, there exists $(f_1',\ldots,f_m')\in V^m$ with $(g,f_1',\ldots,f_m')\in A$. As 
$N\unlhd A$, 
$$(1,f_1,\ldots,f_m)^{(g,f_1',\ldots,f_m')}=(1,f_1'^{-1}f_1^gf_1',\ldots,f_m'^{-1}f_m^{g}f_m')\in
N.$$ 
By hypothesis, we have
$1=(f_i'^{-1}f_i^{g}f_i')(\sigma)=f_i'(\sigma)^{-1}f_i(\sigma^{(g^\pi)^{-1}})f_i'(\sigma)=f_i(\mu)$.
Since $\mu$ is an arbitrary element of $\Delta$ we obtain $f_i=1$. Since $(1,f_1,\ldots,f_m)$ was an
arbitrary element of $N$, it follows that $G(i)$ and $H(i-1)$ hold.~$_\blacksquare$
\smallskip

\smallskip
\noindent\textsc{Claim 2. }Let $i\in \{1,\ldots,m\}$. Then $G(i)\Longleftrightarrow   H(i-1)$.
\smallskip

\noindent 
Suppose that $G(i)$ holds. Applying Claim~1 with $\sigma=\delta$, we immediately obtain $H(i-1)$.
Conversely, if $H(i-1)$ holds then applying Claim~1 with some $\sigma\in\Delta\setminus\{\delta\}$,
we obtain $G(i)$.~$_\blacksquare$

\smallskip
\noindent\textsc{Claim 3. }Let $i\in \{0,\ldots,m-1\}$ and let $j$ be in the $R$-orbit of $i$. Then
$G(i)\Longrightarrow G(j)$ and $H(i)\Longrightarrow H(j)$.
\smallskip

\noindent 
Assume that $G(i)$ holds and let $n=((g_0,\ldots,g_m),(h_0,\ldots,h_{m-1}))\in N$. There exists
$r\in R$ such that $ir^{-1}=j$. Since $R\leq P_2$, $N$ is normalised by $R$ and $n^r\in N$.
By~$(\ref{LA})$, this implies that $((g_{0r^{-1}},g_{1 r^{-1}}, \dots, g_{m r^{-1}}),
(h_{0r^{-1}},\dots,h_{(m-1)r^{-1}}))\in N$. Since $G(i)$ holds, we have that $g_j=g_{ir^{-1}}=1$. As
$n$ was an arbitrary element of $N$, this shows that $G(j)$ holds. The proof that
$H(i)\Longrightarrow H(j)$ is essentially the same and is omitted.~$_\blacksquare$
\smallskip

\smallskip
\noindent\textsc{Claim 4. }$G(i)$ holds for every $i\in\{0,\ldots,m\}$.
\smallskip

\noindent We argue by contradiction and let $z$ be minimal in $\{0,\ldots,m\}$ such that $G(z)$ does
not hold. Since $N\leq M$, we have that $G(0)$ holds and thus $z\geq 1$. By Claim 2, we see that
$H(z-1)$ does not hold.

Let $\mathcal{O}$ be the $R$-orbit on $\{0,\ldots,m\}$ containing $z$. By the minimality of $z$ and
Claim~3, we get that $z$ is the minimum of $\mathcal{O}$. By examining the orbits of $R$ on
$\{0,\ldots,m\}$, we see that this implies that $z-1$ and $z-2$ are in the same $R$-orbit. Since
$H(z-1)$ does not hold, Claim 3 implies that neither does $H(z-2)$. By Claim 2, neither does
$G(z-1)$, contradicting the minimality of $z$. ~$_\blacksquare$
\smallskip

Claim 2 together with Claim 4 implies that $H(i)$ holds for every $i\in\{0,\ldots,m-1\}$ and thus
$N=1$. This concludes the proof.
\end{proof}

Finally, we have $|B|\geq |C| \geq |M|=|K_{\lambda}|^{|\Delta|m}$, where the last equality follows
by Lemma~\ref{lemma:yes}. Recall that $m=\ell m_2$. Since $K$ is not semiregular, we have
$|K_{\lambda}|\geq 2$ and thus $|B|\to\infty$ as  $\ell\to\infty$. This concludes the proof of
Theorem~\ref{theo:main}.

\section{Proof of Theorem~\ref{theo:mainmain}}\label{sec4}
As in the previous section, all groups considered are finite. Suppose first that $L_1,\ldots,L_k$
are regular permutation groups and let $\mathcal{A}=\bigcup_{i=1}^kP_i$ be a rank $k$ faithful
amalgam of permutation type $[L_1,\ldots,L_k]$.  Since $L_i$ is regular, we have
$\mathcal{B}(\mathcal{A})\unlhd P_i$ for every $i\in\{1,\ldots,k\}$. As $\mathcal{A}$ is faithful,
this implies that $\mathcal{B}(\mathcal{A})=1$.  This proves the implication $(2)\Longrightarrow
(1)$ of Theorem~\ref{theo:mainmain}. 

We now turn to the proof of the implication $(1)\Longrightarrow (2)$. The following lemma will be
needed.

\begin{lemma}\label{basic}
Let $H$ and $K$ be transitive permutation groups on $\Delta$ and $\Lambda$, respectively. Let
$\delta_0\in \Delta$, $\lambda_0\in \Lambda$ and let $\ell$ be a positive integer. If
$|\Delta|,|\Lambda|\geq 2$ then there exist a set $\Omega$ of cardinality $\ell |\Delta||\Lambda|$,
faithful group actions  $\rho_H:H\to \Sym(\Omega)$ and $\rho_K:K\to \Sym(\Omega)$, and $\omega\in
\Omega$ such that
\begin{enumerate}
\item $\rho_H(H)_\omega=\rho_H(H_{\delta_0})$ and $\rho_K(K)_\omega=\rho_K(K_{\lambda_0})$;
\item $\langle \rho_H(H),\rho_K(K)_\omega\rangle=\rho_H(H)\times\rho_K(K)_\omega$ and $\langle
\rho_H(H)_\omega,\rho_K(K)\rangle=\rho_H(H)_\omega\times\rho_K(K)$; \label{newlabel}
\item $\langle \rho_H(H),\rho_K(K)\rangle$ is transitive on $\Omega$.
\end{enumerate}
\end{lemma}

\begin{proof}
Let $\Omega$ be the set $\Delta\times \Lambda\times\mathbb{Z}_\ell$ and let
$\omega=(\delta_0,\lambda_0,0)\in\Omega$. Let $g\in \Sym(\Omega)$ be defined by
$$
(\delta,\lambda,i)^g=\begin{cases}
(\delta,\lambda,i)&\textrm{if }\delta\neq \delta_0 \textrm{ or }\lambda\neq \lambda_0,\\
(\delta_0,\lambda_0,i+1)&\textrm{if }\delta=\delta_0 \textrm{ and }\lambda=\lambda_0.\end{cases}
$$
Define $\rho_{H}:H\to\Sym(\Omega)$ by setting $(\delta,\lambda,i)^{\rho_H(h)}=(\delta^h,\lambda,i)$
for every $h\in H$. Similarly, define $\rho_K,\rho_{K}':K\to \Sym(\Omega)$ by setting 
$(\delta,\lambda,i)^{\rho_K'(k)}=(\delta,\lambda^k,i)$ and ${\rho_K(k)}=g^{-1}\rho_K'(k)g$ for every
$k\in K$. It is easy to check that $\rho_H$ and $\rho_K$ define faithful group actions of $H$ and
$K$ on $\Omega$. A simple computation shows that $\rho_H(H)_\omega=\rho_H(H_{\delta_0})$ and
$\rho_K(K)_\omega=\rho_K(K_{\lambda_0})$.

It is easy to check that if $k\in K_{\lambda_0}$ then $\rho_K(k)=\rho'_K(k)$. Since $\rho_H(H)$
centralises  $\rho_K'(K)$, it centralises
$\rho_K'(K_{\lambda_0})=\rho_K(K_{\lambda_0})=\rho_K(K)_\omega$. Similarly, since $H_{\delta_0}$
preserves $\{\delta_0\}$ and $\Delta\setminus\{\delta_0\}$ it follows that $\rho_K(K)$ centralises
$\rho_H(H_{\delta_0})=\rho_H(H)_\omega$. Clearly $\rho_H(H)\cap\rho_K(K)=1$ and hence
$(\ref{newlabel})$ is established.

As $H$ and $K$ are transitive, for every $(\delta,\lambda,i)\in\Omega$, we have 
\begin{eqnarray*}
(\delta,\lambda,i)^{\rho_H(H)\rho_K(K)\rho_H(H)}&=&(\Delta\times\{\lambda\}\times\{i\})^{\rho_K(K)\
rho_H(H)}\\
&\supseteq&((\Delta\setminus\{\delta_0\})\times\{\lambda\}\times\{i\})^{\rho_K(K)\rho_H(H)}\\
&=&((\Delta\setminus\{\delta_0\})\times \Lambda\times\{i\})^{\rho_H(H)}\\
&=& \Delta\times \Lambda\times\{i\}.
\end{eqnarray*}
On the other hand, if $k\in K\setminus K_{\lambda_0}$ then
$(\delta_0,\lambda_0,i)^{\rho_K(k)}=(\delta_0,\lambda_0^k,i-1)$. This shows that $\langle
\rho_H(H),\rho_K(K)\rangle$ is transitive on $\Omega$.
\end{proof}

Let $k$ be a positive integer with $k\geq 3$ and let $L_1,\ldots,L_k$ be nontrivial transitive
permutation groups. For $i\in\{1,\ldots,k\}$, let $m_i$ denote the degree of $L_i$ and denote by
$\{0,\ldots,m_i-1\}$ the set acted upon by $L_i$. (Note that $m_i\geq 2$ since $L_i$ is nontrivial.)
Without loss of generality, we may assume that $L_1$ is not regular and thus $V:=(L_1)_0\neq 1$.

Let $\ell$ be a positive integer. By Lemma~\ref{basic}, there exist faithful actions of $L_2$ and
$L_3$ on a set  $\Omega$ of cardinality $\ell m_2m_3$ with $\langle L_2,L_3\rangle$  transitive on
$\Omega$. Moreover, there exists $\omega_0\in \Omega$ such that $(L_2)_{\omega_0}=(L_2)_0$,
$(L_3)_{\omega_0}=(L_3)_0$, $\langle (L_2)_{\omega_0},L_3\rangle=(L_2)_{\omega_0}\times L_3$ and
$\langle L_2,(L_3)_{\omega_0}\rangle=L_2\times (L_3)_{\omega_0}$.

Let $U=\prod_{\omega\in \Omega}V_\omega$ that is, $U$ is the direct product of $|\Omega|$ copies of
$V$, with the copies indexed by $\Omega$. Observe that the action of $\langle L_2, L_3\rangle$ on
$\Omega$ gives rise to a natural group action of $\langle L_2,L_3\rangle$ on $U$ which enables us to
construct the group $U\rtimes \langle L_2,L_3\rangle$. Let $U'=\prod_{\omega\in
\Omega\setminus\{\omega_0\}}V_\omega$, viewed as a subgroup of $U$ in the natural way. Note that, by
the previous paragraph, $(L_2)_{0}\times(L_3)_{0}$ normalises $U'$. Now, consider the following
abstract groups:  
\begin{eqnarray*}
P_1&:=& L_1\times \left(U'\rtimes ((L_2)_{0}\times(L_3)_{0})\right)\times (L_4)_0\times\cdots \times
(L_{k-1})_0\times (L_k)_{0},\\
P_2&:=& \left(U\rtimes (L_2\times (L_3)_{0})\right)\times (L_4)_0\times\cdots \times
(L_{k-1})_0\times (L_k)_{0},\\
P_3&:=& \left(U\rtimes ((L_2)_{0}\times L_3)\right)\times (L_4)_0\times\cdots \times
(L_{k-1})_0\times (L_k)_{0},\\
P_4&:=& \left(U\rtimes ((L_2)_{0}\times (L_3)_{0})\right)\times L_4\times\cdots \times
(L_{k-1})_0\times (L_k)_0,\\
&\vdots&\\
P_{k-1}&:=& \left(U\rtimes( (L_2)_{0}\times(L_3)_{0})\right)\times (L_4)_0\times\cdots \times
L_{k-1}\times (L_k)_0,\\
P_k&:=& \left(U\rtimes ((L_2)_{0}\times (L_3)_{0})\right)\times (L_4)_0\times\cdots \times
(L_{k-1})_0\times L_k,\\
B&:=&\left(U\rtimes ((L_2)_{0}\times(L_3)_{0})\right)\times (L_4)_0\times\cdots \times
(L_{k-1})_0\times (L_k)_0.
\end{eqnarray*}
Observe that, for every $i\in\{1,\ldots, k\}$, there is an obvious embedding of $B$ in  $P_i$. (For
$i=1$, this is because $U=V\times U'\leq L_1\times U'$.) Hence, in what follows, we regard $B$ as a
common subgroup of $P_1,\ldots,P_k$. Let $\mathcal{A}=\bigcup_{i=1}^kP_i$. Thus $\mathcal{A}$ is a
rank $k$ amalgam of the groups $P_1,\ldots,P_k$ with $\mathcal{B}(\mathcal A)=B$. 

\begin{lemma}\label{new}
The permutation type of $\mathcal{A}$ is  $[L_1,\ldots,L_k]$.
\end{lemma}
\begin{proof}
For every $i\in \{1,\ldots,k\}$, it is immediate from the definitions that the permutation group
induced by the action of $P_i$ on the right cosets of $B$ in $P_i$ is permutation isomorphic to
$L_i$. 
\end{proof}

\begin{lemma}\label{newnew}
The amalgam $\mathcal{A}$ is faithful.
\end{lemma}
\begin{proof}
Let $N$ be a subgroup of $B$ with $N\unlhd P_i$ for every $i\in \{1,\ldots,k\}$. Let $K_i$ denote
the core of $B$ in $P_i$. Clearly, we have 
\begin{eqnarray*}
K_1&=& \left(U'\rtimes ((L_2)_{0}\times(L_3)_{0})\right)\times (L_4)_0\times\cdots \times
(L_{k-1})_0\times (L_k)_{0},\\
K_2&=& \left(U\rtimes (1\times (L_3)_{0})\right)\times (L_4)_0\times\cdots \times (L_{k-1})_0\times
(L_k)_{0},\\
K_3&=& \left(U\rtimes ((L_2)_{0}\times 1)\right)\times (L_4)_0\times\cdots \times (L_{k-1})_0\times
(L_k)_{0},\\
&\vdots&\\
K_k&=& \left(U\rtimes ((L_2)_{0}\times (L_3)_{0})\right)\times (L_4)_0\times\cdots \times
(L_{k-1})_0\times 1,\\
\end{eqnarray*}
and thus $N\leq \bigcap_{i=1}^kK_i= U'$. Let $n\in N$. As $N\leq U$, we may write
$n=\prod_{\omega\in \Omega} n_\omega$ and, since $N\leq U'$, we have $n_{\omega_0}=1$.  Let
$\omega\in \Omega$. Since $\langle L_2,L_3\rangle$ is transitive on $\Omega$, there exists $x\in
\langle L_2,L_3\rangle$ with $\omega^x=\omega_0$. Recall that $\langle L_2,L_3\rangle\leq \langle
P_2,P_3\rangle$ hence $n^x\in N$ therefore $(n^x)_{\omega_0}=1$. On the other hand
$(n^x)_{\omega_0}=n_{\omega_0^{x^{-1}}}=n_{\omega}$. Since this holds for every $\omega\in\Omega$
and every $n\in N$, we have $N=1$ and thus $\mathcal{A}$ is faithful.
 \end{proof}

We have that $|\mathcal{B}(\mathcal A)|=|B|\geq |U|=|V|^{|\Omega|}=|(L_1)_0|^{\ell m_2m_3}$. Since
$L_1$ is not regular, we have $|(L_1)_0|\geq 2$ and thus $|\mathcal{B}(\mathcal A)|\to\infty$ as 
$\ell\to\infty$. This concludes the proof of Theorem~\ref{theo:mainmain}.

\end{document}